\documentclass{amsart}
\usepackage{amssymb, amsmath, latexsym, amsthm}
\usepackage{fancyhdr}
\usepackage{graphicx}

\newtheoremstyle{dotless}{}{}{\itshape}{}{\bfseries}{}{ }{}
\theoremstyle{dotless}
\newtheorem{proposition}{Proposition}[section]
\newtheorem{corollary}{Corollary}[section]
\newtheorem{theorem}{Theorem}[section]
\newtheorem{lemma}{Lemma}%[section]
\theoremstyle{definition}
\DeclareMathOperator{\supp}{supp}

\def\supp{\operatorname{supp}}

\newcommand{\al}{\alpha}
\newcommand{\eps}{\varepsilon}
\newcommand{\la}{\lambda}
\newcommand{\vf}{\varphi}
\newcommand{\om}{\omega}

\newcommand{\cE}{\mathcal E}

\newcommand{\bC}{\mathbb C}

\newcommand{\bI}{\mathbb I}

\newcommand{\bV}{\mathbb V}
\newcommand{\bZ}{\mathbb Z}

\newcommand{\one}{\mathbf{1}}
\newtheorem{remark}[theorem]{Remark}

\begin{document}

\title[Improved surrogate maximum principle]%
{Improved surrogate bi-parameter maximum principle}
\author[P. Mozolyako]{Pavel Mozolyako}
\thanks{PM is supported by the Russian Science Foundation grant 17-11-01064}
\address{Department of Mathematics and Computer Science, Saint Petersburg State University, Saint Petersburg, 199178, Russia}
\email{pavel.mozolyako@unibo.it}
\author[G.~Psaromiligkos]{Georgios Psaromiligkos}
\address{Department of Mathematics, Michigan Sate University, East Lansing, MI. 48823}\email{psaromil@math.msu.edu \textrm{(G.\ Psaromiligkos)}}
\author[A.~Volberg]{Alexander Volberg}
\thanks{AV is partially supported by the NSF grant  DMS 1900268 and by Alexander von Humboldt foundation}
\address{Department of Mathematics, Michigan Sate University, East Lansing, MI. 48823}
\email{volberg@math.msu.edu \textrm{(A.\ Volberg)}}
\author[P.~Zorin-Kranich]{Pavel Zorin-Kranich}
\thanks{PZ was partially supported by the Hausdorff Center for Mathematics (DFG EXC 2047)}
\address[P.~Zorin-Kranich]{Mathematical Institute, University of Bonn, Bonn, Germany}
\email{pzorin@uni-bonn.de}
\keywords{Bi-parameter maximum principle, b-trees, bi-parameter Carleson embedding}
\makeatletter
\@namedef{subjclassname@10010}{
  \textup{2010} Mathematics Subject Classification}
\makeatother
\subjclass[2010]{42B100, 42B35, 42B99, 47A99, 47A100}
\begin{abstract} 
Logarithmic potentials and many other potentials satisfy maximum principle. The dyadic version of logarithmic potential can be easily introduced, it lives on dyadic tree and also satisfies maximum principle. But its analog on bi-tree does not have this property. We prove here that ``on average'' we can still have something like maximum principle on bi-tree.
\end{abstract}
\maketitle

\section{Potentials and maximum principle}
\label{pot}

Let us recall some basic facts on logarithmic potential. For a compact Dirichlet regular set $E$ on the complex plane  (in particular $E$ has positive logarithmic capacity $\text{cap}(E)$) there exists unique positive measure $\mu$  on $E$ such that
$$
\int\log\frac1{|z-\zeta|} d\mu(\zeta) \equiv 1,
\quad z=x+iy\in E\,.
$$
And potential $U^\mu(z):= \int\log\frac1{|z-\zeta|} d\mu(\zeta) $  also satisfies
$$
U^\mu(z) \le 1 \quad \forall z\in \bC\,.
$$
For such a measure, its normalization $\mu/|\mu|$  is called capacitary measure of $E$ and $|\mu|=\text{cap}(E)$.

For any positive measure $\nu$ on $E$ the maximal principle holds: for 
$$
U^\nu(z)=\int\log\frac1{|z-\zeta|} d\nu(\zeta)\quad\,\,\zeta=\xi+i\eta, z= x+iy
$$
 we have
$$
U^\nu(z) \le \sup_{z\in E} U^\nu(z), \quad \forall z\in \bC\,.
$$

However, in the area of complex analysis dealing with Hardy spaces in the poly-disc considered in \cite{AMPS}, \cite{AHMV}, \cite{AMPVZ}, \cite{MPVZ} the following very different potential  (or its dyadic version)
appears absolutely naturally and plays a vital part  of poly-disc theory:
$$
V^\nu (x, y) =\int_E\log\frac1{|x-\xi|} \cdot\log\frac1{|y-\eta|} d\nu(\xi, \eta)\,.
$$
It is a bi-parameter potential. This a very unusual potential and there is no maximum principle. The reader can be familiarized with this kind of potential theory (along with the classical one) through the book \cite{AH}.

In \cite{MPV} we built several examples that  demonstrate how crucially multi-parameter potential theory is different from the usual one. 

But a certain ``shadow'' of the maximum principle is still preserved. We used a certain surrogate maximum principle that turned out to be vital for our multi-parameter Carleson embedding theorems ($\equiv$weighted multi-parameter paraproducts theorems) in \cite{AHMV}, \cite{AMPVZ}, \cite{MPVZ}, \cite{MPVZ1}.

We feel that our method of proving the surrogate maximum principle is a certain variant of a convex optimization. But we did not manage to cast it in this language, and, instead we use trick after trick to get it.

In the current paper we improve the dyadic surrogate maximum principle found in \cite{AMPVZ} and \cite{MPVZ}, \cite{MPVZ1}, so everything happens on a direct product of dyadic trees, we call it a bi-tree and the symbol for it is $T^2$. 

\section{Surrogate maximal principle on a bi-tree}
\label{surr}

%\includepdf[pages=1-2]{OpList.pdf}

Consider a bi-tree $T^2= T\times T$, where $T$ is a finite (but unboundedly large) simple dyadic tree. Let $\al=(\al_1, \al_2)$ be the notation for the node of $T^2$ and
 Hardy  operator be defined
 $$
 \bI f (\al): = \sum_{\al' \ge \al} f(\al'),
 $$
 where $\ge$ is a natural partial order on $T^2$. This is ``summing-up along bi-tree''. The summing-down is a formal conjugate operator:
 $$
 \bI^* f(\al):= \sum_{\al' \le \al} f(\al').
 $$
  Let  $\mu$ be a non-negative function on $T^2$ (on a finite graph there is no difference between non-negative functions and measures).
 The potential $\bV^\mu$ is defined as follows
 $$
 \bV^\mu(\al) := \bI(\bI^* \mu)(\al) =\sum_{\al'\ge \al} \sum_{\gamma \le \al'} \mu(\gamma)\,.
$$
$$
E_s :=\{\al\in T^2: \bV^\mu(\al) \le s\}\,.
$$
More notations:
 $$
 \bV^\mu_\delta(\al) := \bI(\one_{E_\delta}\bI^* \mu)(\al) =\sum_{\al'\ge \al:  \bV^\mu(\al') \le \delta} \,\sum_{\gamma \le \al'} \mu(\gamma)\,.
$$
$$
\cE[\mu] :=\int \bV^\mu d\mu=\int_{T^2} (\bI^*\mu)^2\,.
$$
$$
\cE_\delta[\mu] :=\int \bV^\mu_\delta d\mu=\int_{T^2}\one_{E_\delta} (\bI^*\mu)^2\,.
$$

The same operators exist obviously on the simple dyadic tree, and we call them $I, I^*$. The potential of measure $\mu$ 
on $T$ will be denoted by $V^\mu$. We also define
$$
V^\mu_\delta(a) := \sum_{a'\ge a: \sum_{a''\ge a'} I^*\mu \le \delta} I^*\mu(a')\,.
$$

The maximum principle (at least one variant of it) on a simple tree is the following obvious inequality valid for all non-negative $h$ on a simple tree:
\begin{equation}
\label{maxE}
\max Ih =\max_{\supp h} Ih\,.
\end{equation}
(This is not true in general on a  bi-tree.)

Denoting $h(a'):= (I^*\mu)(a')$ we can use \eqref{maxE} equality to write
\begin{equation}
\label{simpleTREE0}
I(\one_{Ih \le \delta} \cdot h) \le \max_{Ih\le \delta}I(\one_{Ih \le \delta} \cdot h)\le   \delta\,,
\end{equation}
and, hence,  to conclude another obvious thing:
$$
V^\mu_\delta \le \delta
$$
uniformly on a simple tree, and so
\begin{equation}
\label{1param}
\int V^\mu_\delta d\mu \le \delta |\mu|\,.
\end{equation}

Potential $V^\mu$ can be considered as a dyadic version of $\int_0^1 \log\frac1{|x-y|} d\mu(y)$. Notice also that an obvious equality \eqref{maxE} gives the following variant of maximum principle for potentials $V^\mu$:
\begin{proposition}[Maximum principle for potentials on a simple tree]
\label{maxP}
Let $\mu$ be a positive measure on $T$ and $V^\mu\le 1$ on $\supp \mu$. Then $V^\mu \le 1$ everywhere on $T$. 
\end{proposition}
\begin{proof}
%Let $b\in T$ is such that there exists $b'\le b$ such that  $I^*\mu(b') >0$. Then there exists $b''\in \supp\mu, b''\le b'\le b$, Then $V^\mu(b)  \le V^\mu(b') \le 1$.
%If all $b'\le b$ are such that $I^\mu(b')=0$, let $b_0$ be the smallest $\ge b$ such that $I^*\mu(b_0) >0$. Then 
%$V^\mu(b)= V^\mu(b_0)$ (there difference is the sum of $I^*\mu(a), a\in (b, b_0)$, and all those are zero). But if $I^*\mu(b_0) >0$ then there exists $b_1 \le b_0, b_1\in\supp\mu$
%Notice that if $E=\supp h$ is an up-set, then on a simple tree
%\begin{equation}
%\label{maxE}
%\max Ih =\max_{\supp h} Ih\,.
%\end{equation}

Let $h=I^*\mu$, let $E=\supp h$. Then $E=\{a\in T: \exists b\le a, b\in \supp\mu\}$.   By definition $E\subset \{V^\mu\le 1\}=\{ Ih\le 1\}$. 
%In fact, if $a\in E$ then find $b\in\supp\mu$ that corresponds to it, and $V\mu(a) \le V^\mu(b) \le 1$.
By \eqref{maxE} we write $\max Ih =\max_{\supp h} Ih=\max\one_E Ih $, $\one_E Ih \le I(\one_E h)$ (this is true for any up-set on any multi-tree), and so $\le  I(\one_{Ih\le 1} h) \le 1$ by  \eqref{simpleTREE0}. 
\end{proof}

\begin{remark}
\label{counter}
On a bi-tree neither this maximum principle nor \eqref{simpleTREE0} or \eqref{maxE} or \eqref{1param} would be  true anymore.  On can find an example in Section \ref{ce} below or in \cite{MPV}. But on average we have only slightly worse that \eqref{1param}. Namely, we have the following result below.
\end{remark}

\begin{theorem}
\label{eps}
Let $0<|\mu|\le \cE[\mu]<\infty$. Let $\delta>0$ and $\tau>0$. Then there exists $C_\tau<\infty$ such that
$\cE_{\delta}[\mu] \le C_\tau \delta^{1-\tau} |\mu|^{1-\tau} \cE[\mu]^\tau$.
\end{theorem}
First we need a lemma.
\begin{lemma}
\label{lambda}
Let $0<\delta\le \lambda/6$. Then we can find function $\vf$ non-negative on $T^2$ such that
\begin{enumerate}
\item The domain of majorization: $\bI \vf  \ge \frac14\bV^\mu_\delta$, where $\bV^\mu_\delta \ge 40 \lambda$;
\item  Support of majorant: $\supp \vf \subset E_{3\lambda}\setminus E_\delta$;
\item  Energy estimate: $\int_{T^2} \vf^2 \le A_0 \frac{\delta}{3\lambda} \cE_{\delta}[\mu] $.
\end{enumerate}
\end{lemma}

\begin{remark}
\label{instead}
In fact, instead of 1) we will prove the following statement that implies 1).
Everywhere on $T^2$,  we have
$$
\bI \vf + 10\lambda \ge \frac12\bV^\mu_\delta\,.
$$
\end{remark}

\subsection{Discussion of Lemma \ref{lambda}}
Lemma 4.10 of \cite{MPVZ1} claims ``almost'' the same. But there are two very delicate differences. The first difference is that although in Lemma 4.10 the majorization claim 1) 
is present, but in a weaker form: the similar majorization happens there on the set $\{\lambda \le \bV^\mu_\delta \le 2 \lambda\}$. 

The fact that $40 \lambda$ happens instead of $\lambda$ is immaterial, but the fact that Lemma 4.10 has also the estimate from above $\bV^\mu_\delta \le 2 \lambda$ on the domain of majorization makes Lemma 4.10 much weaker than Lemma \ref{lambda} above.

But we should also mention that by restricting majorization condition 1) to a ``smaller'' set  Lemma 4.10  allowed us to have much better  energy estimate 3): $\int_{T^2} \vf^2 \le A_0 \frac{\delta^2}{\lambda^2} \cE_{\delta}[\mu] $.

Such an estimate in our Lemma \ref{lambda} is not possible on a ``larger'' domain of majorization $\bV^\mu_\delta \ge 40 \lambda$ that does not include the upper bound  on potential $\bV^\mu_\delta$ (if we also want to keep the support claim 2)).

The attentive reader should be warned that Lemma 4.10 allows us to get an estimate of type 1) on a set $\bV^\mu_\delta \ge 40 \lambda$ without an upper bound  on potential $\bV^\mu_\delta$, 
and even with the energy estimate $\int_{T^2} \vf^2 \le A_0 \frac{\delta^2}{\lambda^2} \cE_{\delta}[\mu] $. But the price to pay is to throw away the support claim 2) of Lemma \ref{lambda}.

\medskip

We are saying this to emphasize that even though Lemma 4.10 of \cite{MPVZ1} and Lemma \ref{lambda} above look ``the same'', they are actually very different, and they have quite different proofs.

This also explains that on tri-tree $T^3$ we do not have so far an analog of Lemma \ref{lambda}, but we have the analog of Lemma 4.10, that is Lemma 4.17 of \cite{MPVZ1}.

\medskip

What consequences have these subtle delicate differences between Lemmas and $T^2$ and $T^3$? Here they are.
For example, Lemma 4.10 was needed to prove on bi-tree $T^2$ the estimate
$$
\cE_\delta[\mu] \lesssim \delta^{2/3} |\mu|^{2/3} \cE[\mu]^{1/3}\,.
$$
As we can see in Theorem \ref{eps}, this $2/3$ can be replaced by any $1-\tau, \tau>0$. On $T^3$ we do not have the analog of Lemma \ref{lambda} so far, and this is the reason we cannot improve the estimate of Lemma 4.20 of \cite{MPVZ1} on $T^3$:
$$
\cE_\delta[\mu] \lesssim \delta^{1/2} |\mu|^{1/2} \cE[\mu]^{1/2}\,.
$$
Again the reader should compare this with absolutely trivial estimate \eqref{1param} on a simple tree. We do not have any estimate of the type 
$$
\cE_\delta[\mu] \lesssim \delta^{\eps} |\mu|^{\eps} \cE[\mu]^{1-\eps}\,.
$$
with any positive $\eps$ on $T^n$, $n\ge 4$. This prevents us to extend the results of \cite{MPVZ}, \cite{MPVZ1} to higher dimension. In particular, Carleson embedding theorem on $T^n$, $n\ge 4$, is not known.

\medskip

\subsection{From Lemma \ref{lambda} to Theorem \ref{eps}}
Lets see how Lemma \ref{lambda} implies Theorem \ref{eps}. By the first part of Lemma  \ref{lambda} we conclude:
$$
\frac{1}{4}\cE_\delta[\mu]  \le 40\lambda|\mu| + \int \bI \vf \, d\mu = 40\lambda|\mu| + \int_{E_{3\lambda}}  \vf \, \bI^*\mu  \le  40\lambda|\mu| +\big(\int \vf^2\big)^{1/2} \cE_{3\lambda}[\mu]^{1/2}\,.
$$
So
$$
\cE_\delta[\mu]  \le  2 A_0^{1/2}\Big(\frac{\delta}{3\lambda}\Big)^{1/2}\big(\cE_\delta[\mu]\big)^{1/2} \cE_{3\lambda}[\mu]^{1/2} +C\lambda|\mu|\,.
$$
We solve this quadratic inequality with respect to $x:=  \big(\cE_\delta[\mu]\big)^{1/2}$ to obtain
$$
 \cE_\delta[\mu] \le 10A_0\Big(\frac{\delta}{3\lambda}\Big) \cE_{3\lambda}[\mu]  + A_1\lambda |\mu|\,.
 $$
We have for any $\delta,\lambda>0$ with $\delta\leq \dfrac{\lambda}{6}$:

\begin{equation}\label{Ened}
 \cE_\delta[\mu] \lesssim \frac{\delta}{\lambda} \cE_{\lambda}[\mu]  + \lambda |\mu|\,.
\end{equation}
 
Let $T>6$ to be specified later (it depends only on $\tau$). We set $\lambda_i:=T^{i}\delta$ for $i\in\mathbb{Z}_+$.
If $\cE_{\delta}[\mu]\leq \lambda_1 |\mu|$ then 

$$\cE_{\delta}[\mu]=\cE_{\delta}[\mu]^{\tau}\cE_{\delta}[\mu]^{1-\tau}\leq C_{\tau} \cE[\mu]^{\tau} (\delta|\mu|)^{1-\tau}$$ as $\lambda_1=T\delta$. This gives the desired estimate with $C_{\tau}:=T^{1-\tau}$.

If $\cE_{\delta}[\mu]>\lambda_1 |\mu|$ then we use a stopping argument. We fix the first $k\in\mathbb{N}$ such that i) $\cE_{\lambda_{k-1}}[\mu]>\lambda_{k}|\mu|$ but ii) $\cE_{\lambda_{k}}[\mu]\leq \lambda_{k+1}|\mu|$. The ``stopping'' happens as $\cE[\mu]$ is finite but $\lambda_k \underset{k \rightarrow \infty}{\rightarrow} \infty$. 

Now, using \eqref{Ened} for $1\leq i\leq k$ we have 

\begin{equation}
\label{Recur}
\cE_{\lambda_{i-1}}[\mu] \lesssim \frac{\lambda_{i-1}}{\lambda_{i}} \cE_{\lambda_{i}}[\mu] + \lambda_{i} |\mu|\lesssim T^{-1}\cE_{\lambda_{i}} [\mu]
\end{equation}
 For the last inequality we used i) : $\lambda_i|\mu|=T^{-1}\lambda_{i+1}|\mu|\leq T^{-1}\cE_{\lambda_{i}}[\mu]$. Next, we multiply these inequalities to get:

$$
\cE_{\delta}[\mu]\leq \Big(\frac{c_0}{T}\Big)^k\cE_{\lambda_{k}} [\mu]\underset{\text{by ii)}}{\leq} c_0^k T\delta |\mu|\,,
$$ 
where $c_0$ is the constant in \eqref{Recur} which we can assume to be $\geq 6$. We choose $T:=c_0^{\frac{1}{\tau}}$. By i) we see that $k\leq \log_T \big(\frac{A}{\delta}\big)$ where $A= \frac{\cE[\mu]}{|\mu|}$. Thus, $c_0^k \leq \big(\frac{A}{\delta}\big)^{\tau}$, which gives the desired result with $C_{\tau}:=T=c_0^{\frac{1}{\tau}}$. 

In fact,
\begin{equation}
\label{Final}
\cE_{\delta}[\mu]\leq  T \Big(\frac{\cE[\mu]}{\delta |\mu|}\Big)^\tau\delta |\mu|= T\delta^{1-\tau} |\mu|^{1-\tau}\cE[\mu]^\tau\,.
\end{equation}

In the previous case the constant $C_{\tau}$ was $T^{1-\tau}$, but obviously this is less than $T$.  Notice also that $\min\limits_{\tau \in (0,1)} c_0^{\frac{1}{\tau}} \delta^{1-\tau} = \delta e^{c\sqrt{\log\frac1\delta}}$. Thus we have
\begin{corollary}
\label{Vmu}
Let $\mu$ be a measure on bi-tree such that $|\mu|\le \cE[\mu]$. Then
$$
\cE_\delta[\mu] \le  \delta e^{c\sqrt{\log\frac1\delta}} \cE[\mu]\,.
$$
\end{corollary}

\begin{remark}
For $\mu$ be such that $|\mu|\le \cE[\mu]$ on a simple tree we have a trivial estimate
$$
\cE_\delta[\mu]:= \int V^\mu_\delta d\mu \le   \delta\cE[\mu]\,.
$$
It just follows from another obvious one parameter claim \eqref{1param}.
\end{remark}

\section{Majorization in Lemma \ref{lambda}}
\label{maj}
Now we prove Lemma \ref{lambda}.

\begin{proof}
Consider $m:=\one_{E_\delta} \bI^* \mu$, and $n:= \bI^* \mu$. 

Also fix $(\beta_0, \al_0) \in T^2$ and let this node be such that
\begin{equation}
\label{20lambda}
\bV^\mu_\delta(\beta_0,\al_0) \ge 20 \lambda\,.
\end{equation}

Now put
$$
\vf = \frac1\lambda I_1 m \cdot I_2 n\cdot \one_{E_{3\lambda}\setminus E_\delta} + \frac1\lambda I_2 m \cdot I_1 n\cdot \one_{E_{3\lambda}\setminus E_\delta} \,.
$$
Then 
$$
\bI\vf (\beta_0,\al_0) = \frac1\lambda\big(I_2I_1(  I_1 m \cdot I_2 n\cdot \one_{E_{3\lambda}\setminus E_\delta} ) + I_1I_2( I_2 m \cdot I_1 n\cdot \one_{E_{3\lambda}\setminus E_\delta} )  \big) (\beta_0,\al_0)\,.
$$
And 
$$
I_2I_1(  I_1 m \cdot I_2 n\cdot \one_{E_{3\lambda}\setminus E_\delta} ) (\beta_0,\al_0)=\sum_{\al\ge \al_0}I_1(  I_1 m \cdot I_2 n\cdot \one_{E_{3\lambda}\setminus E_\delta} ) (\beta_0,\al)=
$$
$$
\sum_{\al\ge \al_0}I_1 m (\beta_0, \al)\cdot I_1 (N \one_{\delta\le I_1 N \le 3\lambda})(\beta_0, \al),
$$
where $N:= I_2 \, n$. But
\begin{equation}
\label{N}
 I_1 (N \one_{\delta<  I_1 N \le 3\lambda})(\beta_0, \al) \ge \frac{3}{2}\lambda-2\delta\ge \lambda,
 \end{equation}
 %if  there exists $\beta(\al)\ge \beta_0$ such that $\bI \, n(\beta(\al), \al)= (I_1N) (\beta(\al), \al) > 3\lambda$. We denote the set of such $\al\ge \al_0$ by
% $L(\beta_0, \al_0)$. Let $\al\ge \al_0$ be in $L(\beta_0, \al_0)$, and $\al\ge \al'\ge \al_0$. Then $\al' \in L(\beta_0, \al_0)$ as well.
 %In fact,  $(I_1N ) (\beta(\al), \al')  \ge (I_1N) (\beta(\al), \al)  \ge 3\lambda$, because $I_1 N=\bI \,n$ is monotone increasing in each variable. So the set $L(\beta_0, \al_0)$ is the {\it ray}, or {\it segment} $[\al_0, \ell(\beta_0, \al_0)]$.
 if   $\bI \, n(\beta_0, \al)= (I_1N) (\beta(\al), \al) > 3\lambda$. This  set is non-empty by \eqref{20lambda}. We denote the set of such $\al\ge \al_0$ by
 $L(\beta_0, \al_0)$. So $(\beta_0, \al)\notin E_{3\lambda}$. Let $\al\ge \al_0$ be in $L(\beta_0, \al_0)$, and $\al\ge \al'\ge \al_0$. Then $\al' \in L(\beta_0, \al_0)$ as well.
 In fact,  $(I_1N ) (\beta_0, \al')  \ge (I_1N) (\beta_0, \al)  > 3\lambda$, because $I_1 N=\bI \,n$ is monotone increasing in each variable. So the set $L(\beta_0, \al_0)$ is the {\it ray}, or {\it segment} $[\al_0, \ell(\beta_0, \al_0)]$.
 
 To check \eqref{N} let us denote by $\beta(3\lambda)$ the place such that it is the first $>\beta_0$ such that $(\beta(3\lambda), \al) \in E_{3\lambda}$.  If it exists. If it does not exists we put $\beta(3\lambda):= \beta_0$. We also denote by $\beta(\delta)$ the last $>\beta_0$ such that $(\beta(\delta), \al) \notin E_{\delta}$.  
So it may happen that  $\beta(\delta):= O=[0,1]$, the maximal dyadic interval.
 
 Then we get \eqref{N} as follows %(see Figure below):
 $$
 I_1 (N \one_{\delta<  I_1 N \le 3\lambda})(\beta_0, \al) =\sum_{\beta(3\lambda) \le \beta'\le \beta(\delta)} N(\beta', \al) = 
 $$
 $$
 \sum_{\beta(3\lambda) \le \beta'} N(\beta', \al) - \sum_{\beta(\delta) < \beta'} N(\beta', \al) \ge \frac32\lambda   -2\delta\,.
 $$

 \bigskip
 
% \includepdf{Figure1.pdf}
 
 Using \eqref{N} we continue:
 $$
 I_2I_1(  I_1 m \cdot I_2 n\cdot \one_{E_{3\lambda}\setminus E_\delta} ) (\beta_0,\al_0) \ge \lambda \sum_{\al\ge \al_0, \al\in L}I_1 m (\beta_0, \al)\,.
 $$
 
 Symmetrically, if $R(\beta_0, \al_0)$ is the set of $\beta$ such that there  exist $\al(\beta)\ge \al_0$ such that $\bI \, n(\beta, \al(\beta))= I_2 H (\beta, \al(\beta)) \ge 3\lambda$, where $H:= I_1\, n$, we have
 
 $$
 I_1I_2(  I_2 m \cdot I_1 n\cdot \one_{E_{3\lambda}\setminus E_\delta} ) (\beta_0,\al_0) \ge \lambda \sum_{\beta\ge \beta_0, \beta\in R}I_2 m (\beta, \al_0).
 $$
 
 We also conclude as before that  the set $R(\beta_0, \al_0)$ is the {\it ray}, or {\it segment} $[\beta_0, r(\beta_0, \al_0)]$.
 
 \bigskip
 
 But we assumed in \eqref{20lambda} that
 \begin{equation}
 \label{V20}
 I_1I_2 m (\beta_0, \al_0)= I_2I_1 m (\beta_0, \al_0) = \bV^\mu_\delta (\beta_0, \al_0) \ge 20\lambda\,.
 \end{equation}
 
 This assumption allows us to see that
 \begin{equation}
 \label{maj1}
 \sum_{\al\ge \al_0, \al\in L(\beta_0, \al_0)}I_1 m (\beta_0, \al) + \sum_{\beta\ge \beta_0, \beta\in R(\beta_0, \al_0)}I_2 m (\beta, \al_0) \ge  \bV^\mu_\delta (\beta_0, \al_0)  -\delta\ge \frac{19}{20} \bV^\mu_\delta (\beta_0, \al_0)\,.
 \end{equation}
 
 To see \eqref{maj1} let us rewrite it as follows
 
 \begin{equation}
 \label{maj2}
 \sum_{\al_0\le \al\le \ell(\beta_0, \al_0)}I_1 m (\beta_0, \al) + \sum_{\beta_0\le \beta\le r(\beta_0, \al_0)}I_2 m (\beta, \al_0) \ge  \bV^\mu_\delta (\beta_0, \al_0)  -\delta\ge \frac{19}{20} \bV^\mu_\delta (\beta_0, \al_0)\,.
 \end{equation}
 
 To prove \eqref{maj2} let us assume first that
 \begin{equation}
 \label{above}
 (\ell(\beta_0, \al_0), r(\beta_0, \al_0)) \in E_\delta\,.
 \end{equation}
 Recall that $\bV^\mu_\delta = I_2I_1 m$. Thus the first sum in \eqref{maj2}  gives us the summation of $\bI^*\mu$ over the part
  of  $E_\delta$ that is $E_\delta\cap \{(\beta, \gamma)\ge (\beta_0, \al_0): \al \le \ell(\beta_0, \al_0)\}$.
  
  But we symmetrically have $\bV^\mu_\delta = I_1I_2 m$. Thus the second sum in \eqref{maj2}  gives us the summation of $\bI^*\mu$ over the part
  of  $E_\delta$ that is $E_\delta\cap \{(\beta, \gamma)\ge (\beta_0, \al_0): \beta \le r(\beta_0, \al_0)\}$. 
  
For the sake of brevity let us denote $\ell:= \ell(\beta_0, \al_0), r= r(\beta_0, \al_0)$.   The only part of the summation of $\bI^*\mu$ involved in the definition of $\bV^\mu_\delta (\beta_0, \al_0)$, which is left uncovered by both sums of \eqref{maj2} is, therefore, $\sum_{\beta\ge \ell, \, \al\ge r} \bI^*\mu (\beta, \al)$. But by assumption \eqref{above} and the definition of $E_\delta$ this sum is at most $\delta$. Thus \eqref{maj2} is proved when \eqref{above}  holds. %See Figure 2 below.

\bigskip

%\includepdf{Figure2.pdf}

%\includepdf{Figure3.pdf}

Now assume that \eqref{above} does not hold. In this case we are going to estimate $ \bV^\mu_\delta (\beta_0, \al_0)$ from above and to come to contradiction with \eqref{V20}. In fact, $ \bV^\mu_\delta (\beta_0, \al_0)$ is bounded by three sums:
\begin{itemize}
\item $\sum_{\beta_0\le\beta<r,\, \al_0\le \al<\ell} \one_{E_\delta}(\beta, \al)\bI^*\mu(\beta, \al)$, 
\item$\sum_{\beta_0\le\beta, \al\ge \ell} \one_{E_\delta}(\beta, \al)\bI^*\mu(\beta, \al)$, 
\item$\sum_{\al_0\le \al, \beta\ge r} \one_{E_\delta}(\beta, \al)\bI^*\mu(\beta, \al)$.
\end{itemize}
The first sum vanishes as by the fact that we have negation of \eqref{above} each term of this sum is zero. In fact, let 
$$
(\beta, \al) \in [\beta_0, r]\times [\al_0, \ell],
$$
then $(\beta, \al) \notin E_\delta$ because by the negation of \eqref{above} $(r, \ell)\notin E_\delta$. Thusl $\one_{E_\delta}(\beta, \al)=0$  in the first sum. 

The second sum $\le 6\lambda$. This is because $\ell$ is the maximal element for which such a sum is $\ge  3\lambda$. This maximality, and the monotone increasing  of $\al\to I_1 m (\beta_0, \al)$, imply that the second sum is at most $6\lambda$. The monotone increasing of $\al\to I_1 m (\beta_0, \al)$ follows from the monotone increasing of function $m=\one_{E_delta} \bI^* \mu$ in both variables.

By the same reasoning (symmetrically) the third sum is at most $6\lambda$.

We conclude that if \eqref{above} does not hold then  $ \bV^\mu_\delta (\beta_0, \al_0) \le 12\lambda$, but this contradicts \eqref{V20}. We finally proved \eqref{maj2}, and, therefore, \eqref{maj1}, under the assumption \eqref{20lambda} that  $ \bV^\mu_\delta (\beta_0, \al_0)\ge 20\lambda$. %See Figure 3 above.
 
 \bigskip

 From \eqref{maj1} it follows now that
 $$
 \bI \vf (\beta_0, \al_0) \ge  \frac{19}{20} \bV^\mu_\delta (\beta_0, \al_0),
 $$
 if \eqref{V20} is satisfied. Obviously this gives the majorization claim in Lemma. The support condition follows by the definition of $\vf$.
 
 \section{The energy estimation}
 \label{energy}
 
 We are left to prove the norm (energy) claim of Lemma \ref{lambda}:
 \begin{equation}
 \label{en}
 \int_{T^2} \vf^2 \le K \frac{\delta}{\lambda} \cE_{\delta}[\mu]\,.
 \end{equation}
 By the definition of $\vf$, 
 $$
  \int_{T^2} \vf^2 = \frac2{\lambda^2}\int (I_1m )^2 \cdot (I_2\, n)^2 \one_{E_{3\lambda}} + \frac2{\lambda^2}\int (I_2m )^2 \cdot (I_1\, n)^2 \one_{E_{3\lambda}}\,.
  $$
  These two terms are symmetric, we will estimate the first one. To do that we need two lemmas.
  \begin{lemma}
  \label{2.1}
  Let $T$ be a finite dyadic tree, and $g, h$ be non-negative functions on $T$. Let $g$ be superadditive, and  let $Ih \le \lambda$ on $\supp g$. Then  for any $\beta\in T$
  $$
  I^*(gh)(\beta) = \sum_{\al\le \beta} g(\al) h(\al) \le \la g(\beta)\,.
  $$
  \end{lemma}
  \begin{proof}
  Let us prove that
   $$
  I^*(gh)(\beta) = \sum_{\al\le \beta} g(\al) h(\al) \le g(\beta) \max_{\al\in T, \al\le \beta} Ih(\al) \,.
  $$

  The support of $g$ is an up-set by superadditivity. Then this holds trivially if $g(\beta)=0$, and so we  need to check the claim only on the support of $g$.
  Let $\beta\in \supp g$ and let $\beta_+, \beta_-$ be two children of $\beta$. 
  Then by induction
  \begin{align*}
  & I^*(gh)(\beta) = g(\beta) h(\beta) + I^*(gh)(\beta_+) + I^*(gh)(\beta_-)  \le
  \\
  &  g(\beta) h(\beta) +  g(\beta_+) \max_{\al\in T, \al\le \beta_+} Ih(\al) +g(\beta_-) \max_{\al\in T, \al\le \beta_-} Ih(\al) \le
  \\
  & g(\beta) h(\beta) + (g(\beta_+) + g(\beta_-))\cdot \max [ \max_{\al\in T, \al\le \beta_+} Ih(\al),  \max_{\al\in T, \al\le \beta_-} Ih(\al)] \le
  \\
  &g(\beta) h(\beta) + g(beta) \cdot \max [ \max_{\al\in T, \al\le \beta_+} Ih(\al),  \max_{\al\in T, \al\le \beta_-} Ih(\al)] =
  \\
  & g(\beta)\cdot [ h(\beta) +  \max_{\al\in T, \al< \beta} Ih(\al)] =  g(\beta)\cdot [  \max_{\al\in T, \al\le \beta} Ih(\al)] 
  \end{align*}
  Lemma \ref{2.1} is proved.
  \end{proof}
 
 In the next lemma the operator $I$  is an operator  on any abstract space with any positive measure. 
 \begin{lemma}
 \label{2.2}
 Let $I$ be an operator with positive kernel, and $f, g$ non-negative functions. Then
 $$
 \int (If)^2 g \le \sup_{\supp \,g} II^*(g) \int f^2
 $$
 \end{lemma}
 \begin{proof}
 Let $Ih(x)=\int K(x, y) h(y) dy$, here $dy$ is just any measure. Then $I^* f(y) =\int K(x, y) f(x) dx$. Then
 \begin{align}
 \label{conj}
 & \int (If)^2 g = \int f I^*\big( g If\big) \le \|f\|_2 \Big( \int  I^*\big( g If\big)  I^*\big( g If\big) dy\Big)^{1/2}.
 \end{align}
 On the other hand, we can write (we use symmetry between $x$ and $x'$ in the third line)
 \begin{align*}
 &\int  I^*\big( g If\big)  I^*\big( g If\big) dy = \int dy\int dx'\int dx K(x,y) g(x) If(x) K(x', y)   g(x') If(x') \le
 \\
 &\frac12  \int dy\int dx'\int dx K(x,y) g(x) \big[ If(x)^2 + If(x')^2\big] K(x', y)   g(x')  =
 \\
 & \int dy\int dx'\int dx K(x,y) g(x)  If(x)^2  K(x', y)   g(x') =\int I^* g(y)\cdot  I^*(g (If)^2)(y) dy=
 \\
 &
 \int II^* g \cdot g \cdot (If)^2 \le \sup_{\supp\, g} II^* g \cdot \int (If)^2 g\,.
 \end{align*}
 Now we plug this estimate into \eqref{conj} and we get
 $$
  \int (If)^2 g  \le \Big(\int f^2\Big)^{1/2}\cdot \Big( \sup_{\supp\, g} II^* g \Big)^{1/2}\cdot \Big(\int (If)^2 g\Big)^{1/2}\,.
  $$
  This is exactly the claim of Lemma \ref{2.2}.
 \end{proof}
 \begin{remark}
 There is always a stronger version:
 $$
 \int (If)^2 g \le \sup_{\supp \,g} I(\one_{\supp\, f}I^*(g)) \int f^2.
 $$
 In fact, we just apply Lemma \ref{2.2} to a new operator $\tilde I\phi := I(\one_{\supp\, f}\phi)$.
 \end{remark}
  Then let $f:=m, I:= I_1, g:= (I_2\, n)^2 \one_{E_{3\lambda}}$, then from this lemma and this remark we conclude the following:
 \begin{equation}
 \label{first}
 \int (I_1m )^2 \cdot (I_2\, n)^2 \one_{E_{3\lambda}} \le  \sup I_1(\one_{\supp m} I_1^*[ \rho] )  \int m^2, 
 \end{equation}
 where $\rho:= (I_2\, n)^2 \one_{E_{3\lambda}}$.  
 
 Now we apply Lemma \ref{2.1} with $g:= (I_2 n)\cdot  \one_{E_{3\lambda}}, h:= (I_2\, n)$, $I:=I_1$.
 Notice
 \begin{enumerate}
 \item $I_1 I_2 n=  \bI n = \bI\bI^*\mu= \bV^\mu$.
 \item So $E_{3\lambda} =\{ \bV^\mu \le 3\lambda\} = \{ I_1 h \le 3\lambda\}$, as $ h:= (I_2\, n)$.
 \item Hence, $\supp g \subset E_{3\lambda} =  \{ I_1 h \le 3\lambda\}$.
 \end{enumerate}
 Therefore, we are in the assumptions of Lemma \ref{2.1} with $g= I_2 n\cdot  \one_{E_{3\lambda}}$, $h=I_2 n\cdot  \one_{E_{3\lambda}}$, $\supp g\subset E_{3\la} =\{ I_1(I_2 n)\le 3\la\}$. We conclude that the following pointwise estimate holds: 
 $$
I_1^* \rho =  I_1^*\big( (I_2 n)^2\cdot  \one_{E_{3\lambda}}\big) \le 3\lambda\cdot I_2n\,.
 $$
 
 \medskip
 
 \begin{remark}
 \label{I*lemma}
Lemma \ref{2.1} was vital here. This is the only place we used that $I_1$ is a simple tree operator.
\end{remark}

\medskip

 But $\supp m \subset E_\delta$. So we get an  estimate
 \begin{equation}
 \label{second}
I_1(\one_{\supp m} I_1^* \rho) \le  I_1 (\one_{E_\delta} (I_1^*\rho)) \le 3\lambda \cdot  I_1(\one_{E_\delta} I_2n) \le   3\delta\lambda,
 \end{equation}
 where the last estimate follows by the following simple observation. We denoted $h=I_2 n$ and we know that (see (1), (2) above with $\delta$ replacing $3\lambda$):
 $$
 E_\delta = \{ I_1 h \le \delta\}\,.
 $$
 The latter set is of course an up-set on a simple tree. Now on a simple tree (but not on a multiple tree)
 \begin{equation}
 \label{simpleTREE}
I_1(\one_{E_\delta}\cdot I_2n) = I_1 (\one_{I_1 h\le \delta}\cdot h) \le \delta\,.
 \end{equation}
 
 Hence, plugging \eqref{second} into \eqref{first}, we conclude that
 $$
 \int_{T^2} \vf^2 \le A_0 \frac{\delta}{3\lambda} \int m^2,
 $$
  but 
 $$
 \int m^2 =\int \one_{E_\delta} [\bI^*\mu]^2= \int \bI [\one_{E_\delta} \bI^*\mu]  d\mu =\int\bV^\mu_\delta d\mu =\cE_\delta[\mu]\,,
 $$
 and  Lemma \ref{lambda} is proved.

\end{proof}

\section{A shorter proof of Theorem \ref{eps} given Lemma \ref{lambda}}
\label{shorter}

\begin{theorem}
\label{thm:surr-mar}
Let $\mu$ be a measure on $T^{2}$.
If $\cE[\mu] \geq 2\delta |\mu|$, then
\[
\cE_\delta[\mu]
\lesssim
\delta \exp\Bigl(c\sqrt{\log \frac{\cE[\mu]}{\delta |\mu|}} \Bigr) |\mu|.
\]
\end{theorem}

\begin{remark} Corollary \ref{Vmu} is a consequence of this result. This is immediate if one notices that function $x\to \frac1{x} e^{\sqrt{\log x}}$ is decreasing for $x\ge e$.
\end{remark}

\begin{remark}
Theorem \ref{eps} is the consequence of this result. In fact,   for any $\tau>0$
$$
\sqrt{ \log x } \le \tau \log x + c_\tau,\quad \forall x\ge 2\,.
$$
\end{remark}

\begin{proof}[Proof of Theorem~\ref{thm:surr-mar} assuming Lemma~\ref{lambda}]
By Lemma~\ref{lambda}, we have
$$
\cE_\delta[\mu]  \le 10\lambda|\mu| + \int \bI \vf \, d\mu = 10\lambda|\mu| + \int_{E_{3\lambda}}  \vf \, \bI^*\mu  \le  10\lambda|\mu| +\big(\int \vf^2\big)^{1/2} \cE_{3\lambda}[\mu]^{1/2}\,.
$$
So
$$
\cE_\delta[\mu]  \le  A_0^{1/2}\Big(\frac{\delta}{3\lambda}\Big)^{1/2}\big(\cE_\delta[\mu]\big)^{1/2} \cE_{3\lambda}[\mu]^{1/2} +10\lambda|\mu|\,.
$$
We solve this quadratic inequality with respect to $x:=  \big(\cE_\delta[\mu]\big)^{1/2}$ to obtain
\begin{equation}
\label{eq:cE-induction-step}
\cE_\delta[\mu] \le 2 A_0\Big(\frac{\delta}{3\lambda}\Big) \cE_{3\lambda}[\mu]  + 10\lambda |\mu|\,.
\end{equation}

Let us denote
\begin{equation}
\label{A}
A= \frac{\cE[\mu]}{|\mu|}\,.
\end{equation}
Let $\delta_{k} := A (4A_{0})^{-k(k+1)/2}$.
By induction on $k\geq 0$, we will show that
\[
\cE_{\delta_{k}}[\mu]
\leq
(4A_{0})^{k+1} \delta_{k} |\mu|,
\]
which, together with monotonicity of $\cE_{\delta}[\mu]$, implies the conclusion of the theorem.

For $k=0$, the claim holds trivially.
We have
\[
\frac{\delta_{k+1}}{\delta_{k}} = (4A_{0})^{-(k+1)(k+2)/2 + k(k+1)/2} = (4A_{0})^{-k-1} \leq 1/6,
\]
so that, using \eqref{eq:cE-induction-step} and the inductive hypothesis, we obtain
\begin{align*}
\cE_{\delta_{k+1}}[\mu]
&\leq
2 A_0 \frac{\delta_{k+1}}{\delta_{k}} \cE_{\delta_{k}}[\mu]  + 4 \delta_{k} |\mu|.
\\ &\leq
2 A_0 \cdot  \frac{\delta_{k+1}}{\delta_{k}} \cdot (4A_{0})^{k+1} \delta_{k} |\mu| + 4 \delta_{k} |\mu|
\\ &\leq
\Bigl( 2 A_0 (4A_{0})^{k+1} + 4 \cdot (4A_{0})^{k+1} \Bigr) \delta_{k+1} |\mu|
\\ &\leq
(4A_{0})^{k+2} \delta_{k+1} |\mu|.
\end{align*}
\end{proof}

\section{The lack of maximum principle}
\label{ce}

All measures and dyadic rectangles  below will be $N$-coarse.

In this section we build another example when Carleson condition holds, but restricted energy condition fails. 
But the example is more complicated (and more deep) than the previous one. In it the weight $\al$  again has values either $1$ or $0$, but the support $S$ of $\al$ is an {\it up-set}, that is, it contains every ancestor of every rectangle in $S$. 

Let $N$ below be $2^M$, $M\in \bZ_+$.
The example is based on the fact that potentials on bi-tree may not satisfy maximal principle. So we start with constructing $N$-coarse $\mu$ such that given a small $\delta>0$
\begin{equation}
\label{le1}
\bV^\mu \lesssim \delta\quad \text{on}\,\, \supp\mu,
\end{equation}
but  with an absolute strictly positive $c$
\begin{equation}
\label{logN}
\max \bV^\mu \ge \bV^\mu(\om_0) \ge c\, \delta\log N\,,
\end{equation}
where $\om_0:=[0, 2^{-N}]\times [0, 2^{-N}]$.

We define a collection of rectangles
\begin{equation}\label{e:773}
Q_j :=  [0,2^{-2^j}] \times  [0,2^{-2^{-j}N}] ,\quad j=1\dots M\approx \log N,
\end{equation}
and we let 
\begin{equation}\label{Qs}
\begin{split}
&Q_j^{++} :=  [2^{-2^j-1},2^{-2^j}]\times  [2^{-2^{-j}N-1},2^{-2^{-j}N}] \\
&Q_j^\ell := [0,2^{-2^j-1}] \times  [0,2^{-2^{-j}N}] ,\quad j=1\dots M\\
&Q_j^r := Q_j\setminus Q_j^\ell\\
&Q_j^t:= [0,2^{-2^j}]\times  [2^{-2^{-j}N-1},2^{-2^{-j}N}] \\
& Q_j^{--} := Q_j^\ell\setminus Q_j^t
\end{split}
\end{equation}
to be their upper right quadrants, lower halves, top halves, right halves, and lower quadrant respectively.
Now we put 
\begin{equation}\label{e:774}
\begin{split}
&\mathcal{R} := \{R:\; Q_j\subset R\; \textup{for some }j=1\dots M\}\\
&\alpha_{Q} := \chi_{\mathcal{R}}(Q)\\
%&\mu(\omega_0) := \frac{1}{MN}\\
&\mu(\omega) := \frac{\delta}{N}\sum_{j=1}^M \frac{1}{|Q_j^{++}|}\chi_{Q_j^{++}}(\omega),\\
%&F := \bigcup_{j=1}^M Q^-_j,
& P_j= (2^{-2^j}, 2^{-2^{-j}N})\,.
\end{split}
\end{equation}
here $|Q|$ denotes the total amount of points $\omega \in (\partial T)^2\cap Q$, i.e. the amount of the smallest possible rectangles (of the size $2^{-2N}$) in $Q$.\par
%We claim that the measure, 
%weight and set defined above satisfy the claim of Proposition.
Observe that on $Q_j$ the measure is basically a uniform distribution of the mass $\frac{\delta}{N}$ over the upper right quarter $Q_j^{++}$ of the rectangle $Q_j$ (and these quadrants are disjoint).
\par

To prove \eqref{le1} we fix $\om\in Q_j^{++}$ and split $\bV^\mu(\om)= \bV^\mu_{Q_j^{++}}(\om) +\mu(Q_j^t)+ \mu(Q_j^r)+ \bV^\mu(Q_j^{++})$, where the first term sums up $\mu(Q)$ for $Q$ between $\om$ and $Q_j^{++}$. This term obviously  satisfies $\bV^\mu_{Q_j^{++}}(\om) \lesssim \frac{\delta}{N}$. Trivially $\mu(Q_j^t)+ \mu(Q_j^r)\le  \frac{2\delta}{N} $. The non-trivial part is the estimate
\begin{equation}
\label{VQ}
\bV^\mu(Q_j^{++}) \lesssim \delta\,.
\end{equation}
To prove \eqref{VQ}, consider the sub-interval of interval $[1, n]$ of integers. We assume that $j\in [m, m+k]$.  We call by $C^{[m, m+k]}_j$ the family of dyadic rectangles containing $Q_j^{++}$ along with all $Q_i^{++}$, $i\in [m, m+k]$ (and none of the others). Notice that $C^{[m, m+k]}_j$ are not disjoint families, but this will be no problem for us as we  wish to
estimate $\bV^\mu(Q_j^{++})$ from above. 

Notice that, for example, $C^{[m, m+1]}_j$ are exactly the dyadic rectangles containing point $P_j$. It is easy to calculate that the number of such rectangles
is 
\[
(2^j+1)\cdot (2^{-j}N +1)\lesssim N\,.
\]

Analogously, dyadic rectangles in family $C^{[m, m+k]}_j$ have to contain  points $P_m, P_{m+k}$. Therefore,  each of such rectangles contains  point $(2^{-2^m}, 2^{-2^{-m-k}N})$.
The number of such rectangles is obviously at most $\lesssim 2^{-k}N$. The number of classes $C^{[m, m+k]}_j$ is at most $k+1$.

Therefore, $\bV^\mu((Q_j^{++})$ involves at most $(k+1)2^{-k}N$ times the measure in the amount $k\cdot \frac{\delta}{N}$. Hence
\[
\bV^\mu((Q_j^{++}) \le \sum_{k=1}^n k(k+1)2^{-k}N \cdot \frac{\delta}{N}\,,
\]
and \eqref{VQ} is proved. Inequality \eqref{le1} is also proved.

\medskip

We already denoted
\[
\om_0:=[0, 2^{-N}]\times [0, 2^{-N}]\,,
\]
 calculate now $\bV^\mu(\om_0)$. In fact, we will estimate it from below. The fact that $C^{[m, m+k]}_j$ are not disjoint may represent the problem now because  we wish estimate $\bV^\mu(\om_0)$ from below.
 
 To be more careful for every $j$ we denote now by $c_j$ the family of dyadic rectangles containing the point $P_j$ but not containing any other point
$P_i, i\neq j$. Rectangles in $c_j$ contain $Q_j^{++}$ but do not contain any of $Q_i^{++}$, $i\neq j$. There are 
$(2^{j} - 2^{j-1}-1)\cdot (2^{-j+1}N-2^{-j}N -1)$, $j=2, \dots, M-2$. This is at least $\frac18 N$.

But now families $c_j$ are disjoint, and rectangles of class $c_j$ contribute at least $\frac18 N\cdot \frac{\delta}{N}$ into the sum that defines $\bV^\mu(\om_0)$. W have $M-4$ such classes $c_j$, as $j=2, \dots, M-2$.
Hence,
\begin{equation}
\label{deltaM}
\bV^\mu(\om_0) \ge \frac18 N\cdot \frac{\delta}{N}\cdot (M-4)\ge\frac19 \delta M\,.
\end{equation}
Choose  $\delta$   to be a small absolute number $\delta_0$. Then we will have (see \eqref{le1})
\[
\bV^\mu \le 1, \quad\text{on}\,\, \supp\mu\,.
\]
But \eqref{deltaM}  proves also \eqref{logN} as $M\asymp \log N$.

\medskip

\begin{remark}
\label{Mx}
Notice that in this example $\bV^\mu\le 1$ on $\supp\mu$, and
\begin{equation}
\label{exp}
\text{cap} \{\om: \bV^\mu\ge \la\} \le c e^{-2\la}\,.
\end{equation}
Here capacity is the bi-tree capacity defined e. g. in \cite{AMPS18}. So there is no maximal principle for the bi-tree potential, but the set, where the maximal principle breaks down,  has small capacity.
\end{remark}

\bigskip

%%%%%%%%%%%%%%%%%%%%%%%%%%%%%%%%%%%%%%%%


\begin{thebibliography}{ZZZZ}




\bibitem[AH]{AH} {\sc R. Adams, L. Hedberg}  {Function Spaces and Potential Theory}, Springer 1999.


\bibitem[AHMV]{AHMV} {\sc Nicola Arcozzi, Irina Holmes, Pavel Mozolyako, Alexander Volberg}. \textit{Bi-parameter embedding and measures with restriction energy condition},  Math. Ann. 377 (2020), no. 1-2, 643--674.

\bibitem[AMPS]{AMPS}{\sc Nicola Arcozzi, Pavel Mozolyako, Karl-Mikael Perfekt, Giulia Sarfatti}. \textit{Carleson measures for the Dirichlet space on the bidisc}, arXiv:1811.04990, pp. 1-44, 2018.

%\bibitem[HPV]{HPV} I. Holmes, G. Psaromiligkos, A. Volberg. \textit{A comparison of box andCarleson conditions on bi-trees}, arXiv:1903.02478, pp. 1--17, 2018.

\bibitem[AMPVZ]{AMPVZ} {\sc Nicola Arcozzi, Pavel Mozolyako, Alexander Volberg, Pavel Zorin-Kranich}. \textit{Bi-parameter Carleson embeddings with product weights}, arXiv:1906.11150, pp. 1-24.

\bibitem[BP]{BP} {\sc A. Barron, J. Pipher},
{\em Sparse domination for bi-parameter operators using square functions},
Preprint, arXiv:1709.05009, 1--22.

\bibitem[Car]{Car} {\sc Lennart Carleson}, {\em A counter example for measures bounded on $H^p$ for the bi-disc}, Preprint (1974).


\bibitem[Ch]{Ch} {\sc Sun-Yang A. Chang}, {\em Carleson measure on the bi-disc}. Ann. of Math. (2) 109 (1979), no. 3, 613--620.

\bibitem[ChF1]{ChF1} {\sc Sun-Yung A. Chang;  Robert Fefferman}, {\em A continuous version of duality of H1 with BMO on the bidisc.} Ann. of Math. (2) 112 (1980), no. 1, 179--201.

\bibitem[ChF2]{ChF2}  {\sc Sun-Yung A Chang;  Robert Fefferman}, {\em  Some recent developments in Fourier analysis and $H^p$-theory on product domains}. Bull. Amer. Math. Soc. (N.S.) 12 (1985), no. 1, 1–43.

\bibitem[RF]{RF} {\sc R. Fefferman,} {\em Harmonic analysis on product spaces}, Ann. of Math., (2), v. 126, 1987, 109--130.



\bibitem[RF1]{RF1} {\sc R. Fefferman},   {\em Calder\'on-Zygmund theory for product domains: $H^p$ spaces}. Proc. Nat. Acad. Sci. U.S.A. v. 83 , no. 4, 1986, 840--843.

\bibitem[RF2]{RF2} {\sc R. Fefferman},   {\em Some recent developments in Fourier analysis and $H^p$ theory on product domains. II}. Function spaces and applications (Lund, 1986), 44–51, Lecture Notes in Math., 1302, Springer, Berlin, 1988. 


\bibitem[GT]{GT}{\sc L. Grafakos,  R. Torres}, {\em Multilinear Calder\'on--Zygmund theory}. Adv. Math. 165, 2002, 124--164.


%\bibitem[TH]{TH}{\sc T. Hanninen}, {\em Equivalence of sparse and Carleson coefficients for general sets}, arXiv:1709.10457.

%\bibitem[IKSTUT]{IKSTUT}
%{\sc A. Iosevich, 
%B. Krause,   E. Sawyer, 
%K. Taylor, I. Uriarte-Tuero}
% {\em Maximal operators: scales, curvature and the fractal dimension}, Anal. Math. v. 45, 2019, 63--86.
     
  \bibitem[JLJ]{JLJ}{\sc J.-L. Journ\'e}, {\em  Two problems of Calder\'on-Zygmund theory on product-spaces}. Ann. Inst. Fourier (Grenoble), v. 38, 1988, no. 1, 111--132. 
  
 \bibitem[JLJ2]{JLJ2}{sc J.-L. Journ\'e},  {\em Caldero\'on--Zygmund operators on product spaces}. Rev. Mat. Iberoamericana 1, 1985, 55--91.

 \bibitem[LSSUT]{LSSUT} {\sc M. T. Lacey,  E.  T.  Sawyer, C.-Y. Shen, I. Uriarte-Tuero}, 
{\em Two-weight inequality for the Hilbert transform: a real variable characterization}, I.
Duke Math. J. 163 (2014), no. 15, 2795--2820. 

\bibitem[L]{L} {\sc M. T. Lacey}, {\em Two-weight inequality for the Hilbert transform: a real variable characterization}, II. Duke Math. J. 163 (2014), no. 15, 2821--2840.

 \bibitem[MPV]{MPV} {\sc P. Mozolyako; G. Psaromiligkos; A. Volberg,} {\em Counterexamples for multi-parameter weighted paraproducts.} C. R. Math. Acad. Sci. Paris 358 (2020), no. 5, 529--534.

\bibitem[MPVZ]{MPVZ} {\sc P. Mozolyako; G. Psaromiligkos; A. Volberg; P. Zorin-Kranich,} {\em Combinatorial property of all positive measures in dimensions 2 and 3.} C. R. Math. Acad. Sci. Paris 358 (2020), no. 6, 721--725. 

\bibitem[MPVZ1]{MPVZ1}{\sc P. Mozolyako; G. Psaromiligkos; A. Volberg; P. Zorin-Kranich,} {\em Carleson embedding in
 on tri-tree and on tri-disc}, preprint, arXiv:2001.02373, pp. 1--33.

\bibitem[MPTT1]{MPTT1}{\sc C. Muscalu, J. Pipher, T. Tao, C. Thiele}, {\em Bi-parameter paraproducts}, Acta Math., 193 (2004), 269--296.

\bibitem[MPTT2]{MPTT2}{\sc C. Muscalu, J. Pipher, T. Tao, C. Thiele}, {\em Multi-parameter paraproducts},   Rev. Mat. Iberoamericana 22 (2006), no. 3, 963–976.

\bibitem[NTV99]{NTV99} F. Nazarov, S. Treil, and A. Volberg. The Bellman functions and two-weight in- equalities for Haar multipliers. J. Amer. Math. Soc. 12,  1999, pp. 909–928.

\bibitem[NTV08]{NTV08} F. Nazarov, S. Treil, A. Volberg, 
Two weight inequalities for individual Haar multipliers and other well localized operators.
Math. Res. Lett. 15 (2008), no. 3, 583--597. 



\bibitem[P]{P}{\sc J. Pipher} {\em Journ\'e's covering lemma and its extension to higher dimensions}, Duke Math. J., 53 , no.  3 (1986), 683--690.



\bibitem[Tao]{Tao}{\sc T. Tao}, {\em Dyadic product $H^1$, $BMO$, and Carleson's counterexample}, preprint, pp. 1--12.
www.math.ucla.edu/~tao/preprints/Expository/product.dvi.

\bibitem[Verb]{Verb} {\sc Igor E. Verbitsky}, {\em embedding and multiplier theorems for discrete Littlewood-Paley spaces}.
Pacific J. Math., v. 176, no. 2, 1996, 529--556.


%\author[N.~Arcozzi]{Nicola Arcozzi}
%\address{Universit\`{a} di Bologna, Department of Mathematics, Piazza di Porta S. Donato, 40126 Bologna (BO)}
%\email{nicola.arcozzi@unibo.it}
%\thanks{Theorem 3.1 was obtained in the frameworks of the project 17-11-01064 by the Russian Science Foundation}
%\thanks{NA is partially supported by the grants INDAM-GNAMPA 10017 "Operatori e disuguaglianze integrali in spazi con simmetrie" and PRIN 10018 "Variet\`{a} reali e complesse: geometria, topologia e analisi armonica"}
%\author[I.~ Holmes]{Irina Holmes}
%\thanks{IH is partially supported by the NSF an NSF Postdoc under Award No.1606270}
%\address{Department of Mathematics, Michigan Sate University, East Lansing, MI. 48823}

% 42B	Harmonic analysis in several variables
% 42B100	Singular and oscillatory integrals (Calder?on-Zygmund, etc.)
% 42B35	Function spaces arising in harmonic analysis
% 47A	General theory of linear operators
% 47A100	Norms (inequalities, more than one norm, etc.)
%{100E100, 47B37, 47B40, 100D55.}
%
% 100D55	$H^p$-classes (1980-10009)
% 100E100	Integration, integrals of Cauchy type, integral representations of analytic functions
%
% 47B   	Special classes of linear operators
% 47B37	Operators on special spaces (weighted shifts, operators on sequence spaces, etc.)
% 47B40	Spectral operators, decomposable operators, well-bounded operators, etc.
%\keywords{Carleson embedding, small energy majorization}





\end{thebibliography}
\end{document}